\RequirePackage[OT1]{fontenc}

\documentclass[journal,twoside,web]{ieeecolor}

\newif\ifarxiv
\arxivtrue   

\usepackage{booktabs} 
\usepackage{generic}
\usepackage{amsmath,amssymb,amsfonts}
\usepackage{algorithmic}
\usepackage{graphicx}
\usepackage{algorithm,algorithmic}
\usepackage{hyperref}
\usepackage{textcomp}

\def\BibTeX{{\rm B\kern-.05em{\sc i\kern-.025em b}\kern-.08em
    T\kern-.1667em\lower.7ex\hbox{E}\kern-.125emX}}
\usepackage{biblatex}
\def\BibTeX{{\rm B\kern-.05em{\sc i\kern-.025em b}\kern-.08em
    T\kern-.1667em\lower.7ex\hbox{E}\kern-.125emX}}
\usepackage{comment}
\usepackage{physics}
\usepackage{amsmath}
\usepackage{amssymb}
\usepackage{mathtools}
\usepackage{amscd}
\usepackage{bbm}
\usepackage{bbold}
\usepackage{mathrsfs}

\usepackage{makecell}

\renewbibmacro*{name:andothers}{
  \ifboolexpr{
    test {\ifnumequal{\value{listcount}}{\value{liststop}}}
    and
    test \ifmorenames
  }
    {\ifnumgreater{\value{liststop}}{1}
       {\finalandcomma}
       {}%
     \andothersdelim\bibstring[\emph]{andothers}}
    {}}

\newtheorem{thm}{Theorem}
\newtheorem{rem}[thm]{Remark}
\newtheorem{cor}[thm]{Corollary}
\newtheorem{lem}[thm]{Lemma}
\newtheorem{defn}[thm]{Definition}

\newcommand{\bx}{\boldsymbol{x}}
\newcommand{\bu}{\boldsymbol{u}}

\newcommand{\Hess}{\mathrm{Hess}}

\newcommand{\id}{\mathbb{1}}

\renewcommand{\hom}{\mathrm{hom}}

\newcommand{\A}{\mathsf{A}}

\newcommand{\rev}[1]{\textcolor{black}{#1}}

\begin{document}
\title{NOMADS: Non-Markovian Optimization-based Modeling for Approximate Dynamics with Spatially-homogeneous Memory}
\author{Ryoji Anzaki and Kazuhiro Sato
\thanks{Ryoji Anzaki was with Digital Design Center, Tokyo Electron Ltd., Tokyo, Japan (e-mail: ryo.anzaki+paper@gmail.com).}
\thanks{Kazuhiro Sato is with Graduate School of Information Science and Technology, The University of Tokyo, Tokyo, Japan (e-mail: kazuhiro@mist.i.u-tokyo.ac.jp).}}

\maketitle

\begin{abstract}
We propose a system identification method,
\emph{Non-Markovian Optimization-based Modeling for Approximate Dynamics with Spatially-homogeneous memory} (NOMADS),
for identifying linear dynamical systems from a set of multi-dimensional time-series data obtained through multiple partially excited experiments.
NOMADS formulates model identification as a convex optimization problem,
in which the state-space coefficient matrices and a memory kernel are estimated jointly
under physically motivated constraints using projected gradient descent.
The proposed framework models memory effects through a
spatially homogeneous kernel, enabling scalable identification of
non-Markovian dynamics while keeping the number of free parameters moderate.
This structure allows NOMADS to integrate information from multiple
multi-dimensional time-series data even when no single experiment provides full excitation.
In the Markovian setting, physical constraints can be incorporated to
enforce conservation laws.
Numerical experiments on synthetic data demonstrate that NOMADS achieves
substantially improved generalization accuracy compared to existing
DMD-based methods even for noisy train data, and reproduces energy conservation in the Markovian case.
\end{abstract}


\section{Introduction}
\label{sec:introduction}


\IEEEPARstart{S}{ystem} identification from experimental time-series data
is a fundamental problem in science and engineering and underpins prediction
and controller design \cite{goodwin2001control}.
In high-dimensional systems, however, it is often infeasible to design
experiments that sufficiently excite all degrees of freedom.
As a result, each experiment typically probes only a subset of the state--input
space, yielding \emph{partially excited} time-series data and complicating
identification and generalization to unseen conditions.

Dynamic Mode Decomposition (DMD) \cite{schmid2010dynamic,tu2014dynamic,schmid2022dynamic}
and its extension DMD with control (DMDc) \cite{proctor2016dynamic}
provide computationally efficient approaches for identifying linear dynamical
systems from data.
Their effectiveness, however, relies critically on the alignment of
excitation between training and test data; when the excited
degrees of freedom in the test data deviate from those explored by
the train data,
the test prediction error can become unbounded.

When applied to partially excited experiments, DMD-based methods suffer from
three fundamental limitations:
\begin{itemize}
  \item[(a)] \textbf{Limited generalization.}
  Models identified from insufficient excitation often exhibit rapidly growing
  prediction errors on unseen test data.
  \item[(b)] \textbf{Lack of physical consistency.}
  Prediction may violate known physical properties, such as
  conservation laws.
  \item[(c)] \textbf{Neglect of non-Markovian dynamics.}
  Standard DMD-based formulations rely on Markovian assumptions and cannot
  adequately represent systems with memory effects or hysteresis.
\end{itemize}

Limitations (a) and (b) arise from partial excitation, and these limitations may persist even when multiple experiments are available, whereas (c) reflects an intrinsic modeling restriction that becomes critical for systems interacting with an external reservoir or environment \cite{svenkeson2016spectral}, such as a thermostat coupled to a heat bath with finite heat capacity.

In this paper, we propose
\emph{Non-Markovian Optimization-based Modeling for Approximate Dynamics with Spatially-homogeneous memory}
(NOMADS), a unified identification framework that addresses these challenges.
NOMADS integrates information from multiple partially excited trajectories,
incorporates memory effects via a low-dimensional memory kernel representation, and enforces
physically motivated constraints within a convex optimization formulation.
This enables consistent identification of global dynamics even when no single
experiment provides full excitation.

The main contributions are as follows:
\begin{enumerate}
  \item We formulate identification of linear dynamical systems with spatially
  homogeneous memory as a convex optimization problem suited to a set of partially
  excited multi-trajectory data.
  \item We develop a projected gradient-based algorithm (NOMADS) and analyze
  its convexity and convergence properties.
  \item Numerical experiments demonstrate improved generalization performance and physical consistency compared with existing DMD-based methods for both Markovian and non-Markovian systems, under noiseless and noisy training conditions.
\end{enumerate}

The remainder of the paper is organized as follows.
Section~\ref{sec:preliminaries} introduces preliminaries,
Section~\ref{sec:optimization} presents the proposed framework,
Section~\ref{sec:convex} discusses the optimization algorithm,
Section~\ref{sec:numerical} reports numerical results, and
Section~\ref{sec:concluding} concludes the paper.



\section{Preliminaries}
\label{sec:preliminaries}



\subsection{\rev{Metric space of tuples of matrices}}
Let us denote the set of real numbers, the set of nonnegative real numbers, set of integers and the set of nonnegative integers by $\mathbf{R}$, $\mathbf{R}_{\geq 0}$, $\mathbf{Z}$, and $\mathbf{Z}_{\geq0}$, respectively.

For $\ell, m\in \mathbf{Z}_{\geq0}$, define an ordered set of integers $[\ell:m]\coloneq \{\ell,\ell+1,\ell+2,\dots,m-1\}\subset \mathbf{Z}_{\geq0}$. We use shorthand notation $[m]\coloneq [0:m]$. Note that $[\ell:m] = \varnothing$ for $\ell \geq m$. 

For a matrix $x\in \mathbf{R}^{n\times m}$, integers $0\leq p < q < n$ and $0\leq s < t < m$, $x_{p:q,s:t} \in \mathbf{R}^{\rev{(q-p)\times (t-s)}}$
is a submatrix. We define the similar notation for vectors.

Let us denote a vector by $\boldsymbol{v}=(v_i)_{i\in[n]}\in\mathbf{R}^n$ where $v_i$ is the $i$th component. We also denote the same vector but first $q$ entries are zero by $(v_i)_{i\in[n]}^q$. Let us denote $m\times s$ matrix filled with $0, 1\in\mathbf{R}$ by $0_{m\times s}$ and $1_{m\times s}$, and $m$ dimensional identity matrix by $\id_{m}$. For a matrix $A$, we denote the transposed matrix by $A^\top$ and the Moore-Penrose pseudoinverse by $A^+$. We also denote $A=(A_{ij})$ where $i, j$ are row and column indices, respectively. If $A$ is a square matrix, we denote the trace (sum of diagonal elements) of the matrix by $\tr(A)$. 

For matrices $A, A'$ of the same shape, let us introduce the inner product by 
\begin{equation}\label{eq:preliminaries:inner:prod:matrix}
  \left\langle A, A' \right\rangle \coloneq \tr(A^\top A').
\end{equation}
This inner product induces the Frobenius norm $\|A\| = \sqrt{\tr(A^\top A)}$. For a pair of tuple of matrices of the same shapes $\A \coloneq \qty(A_0,\cdots,A_{n-1})$ and $\A' \coloneq \qty(A'_0,\cdots,A'_{n-1})$ where $A_i, A'_i \in \mathbf{R}^{n_i\times m_i}$ for $i\in[n]$, we define
\begin{equation}\label{eq:preliminaries:metric}
  \left\langle \A, \A'\right\rangle \coloneq \sum_{i\in[n]}\tr(A_i^\top A_i').
\end{equation}
The inner product induces the norm of a tuple of matrices as follows:
\begin{equation}\label{eq:preliminaries:metric:tuple}
  \left\|\A\right\| \coloneq \sqrt{\left\langle \A, \A\right\rangle}.
\end{equation}


\subsection{Causal Matrices}

In this subsection, we prepare the basic concepts to define linear dynamical systems with homogeneous memory in the next section. The definitions below are useful when dealing with dynamical systems whose state vectors at the first $q+1$ $(q\in\mathbf{Z}_{\geq0})$ time points are given.

\rev{For a function $h: [m] \to \mathbf{R}$ of time $t$, define \textit{time-ordered vector} as $\mathsf{T}_h([m]) = (h(0),h(1), \dots, h({m-1}))$.}

\begin{defn}[Causal matrix]\label{defn:causal}
  A matrix $C \in \mathbf{R}^{m\times m}$ is called causal matrix if $C$ satisfies the following properties for an integer $q\geq0$:
  \begin{itemize}
    \item $C_{:,:q} = 0_{m\times q}$
    \item $C_{q:,q:}$ is an invertible upper-triangular matrix
  \end{itemize}
\end{defn}
A causal matrix is a discrete representation of a causal linear operator defined in \cite{anzaki2023dynamic}. It is called \textit{causal}, because a vector-matrix multiplication with a causal matrix does not violate the causality for a time-ordered vector, i.e., the multiplication results in the weighted sum of \textit{past} elements of the original vector \rev{$\mathsf{T}_h([m])$} for each element: for $t>q$,
\begin{eqnarray}
    (\rev{\mathsf{T}_h([m])}^\top C)_t = \sum_{t'\in [q:t]}C_{tt'}h(t').
\end{eqnarray}

\begin{defn}[Time-invariant causal matrix]
  A matrix $C \in \mathbf{R}^{m\times m}$ is called \textit{time-invariant} causal matrix if $C$ satisfies the following property: in addition to the properties shown in Definition \ref{defn:causal}, there is a $m$-dimensional Toeplitz matrix $T\in\mathbf{R}^{m\times m}$ such that $C_{:,q:} = T_{:,q:}$.
\end{defn}
A time-invariant causal matrix is called time-invariant because $C_{t,s} = C_{t+\delta t, s+\delta t}$.

\begin{defn}
  Assume that $C \in \mathbf{R}^{m\times m}$ is a time-invariant causal matrix with nullity $q$ \rev{(i.e., in this case, the first $q$ columns are zero)}. We say that $C$ is normalized if the diagonal elements of $C_{q:,q:}$ are unity.
\end{defn}

Let us denote the set of $m$-dimensional normalized time-invariant causal matrices with nullity $q$ by $\mathbb{T}_m^q$. Also, we use a shorthand notation 
\begin{equation}\label{eq:preliminaries:unit:of:causal:matrix}
  \id_{m}^q \coloneq \mqty(0_{q\times q} & 0_{q\times (m-q)}\\ 0_{(m-q)\times q} & \id_{m-q}).
\end{equation}

\begin{lem}\label{lemma:preliminaries:pseudoinverse:causal:matrix}
  The left pseudoinverse \footnote{For a matrix $A\in \mathbf{R}^{n\times n}$, a left (right) pseudoinverse $L\in \mathbf{R}^{n\times n}$ ($R\in \mathbf{R}^{n\times n}$) is a matrix that satisfies $LA = \id_n^{q}$ ($AR = \id_n^{q}$) where $q$ is the nullity of $A$.} of $C \in \mathbb{T}_m^q$ is given as follows:
  \begin{equation}
    C^+ = \mqty(0_{q\times q} & 0_{q\times (m-q)}\\ 0_{(m-q)\times q} & (C_{q:,q:})^{-1}).
  \end{equation}
\end{lem}

Note that, for a matrix $C\in\mathbb{T}_m^q$, the left pseudoinverse $C^+$ defined above is not the right pseudoinverse of $C$ in general:
\begin{equation}\label{eq:preliminaries:left:pseudoinverse:from:right}
  CC^+ = \mqty(0_{q\times q} & C_{:q,q:}(C_{q:,q:})^{-1}\\ 0_{(m-q)\times q} & \id_{m-q}).
\end{equation}

We denote the subset of $\mathbb{T}_m^q$ whose elements has bandwidth $Q$ by $\mathbb{T}_m^q(Q)$. Namely, for a matrix $(c_{i,j}) \in \mathbb{T}_m^q(Q)$, $c_{i-k,i}$ is nonzero only if $k \in[\max(Q,i)]$ for $i\in[m]$.

The following theorem is useful to show the convergence of optimization problem for $C\in\mathbb{T}_m^q(Q)$. The proof is straightforward and thus omitted.
\begin{thm} \label{Thm5}
  Assume $Q\geq 0$, $Q \geq q\geq 0$, and $m\geq Q$ are integers. $\mathbb{T}_m^q(Q)$ is a convex and closed subset of $\mathbf{R}^{m\times m}$.
\end{thm}


\section{Problem Settings}
\label{sec:optimization}


\rev{In this section, we formulate the system identification in NOMADS.
We first introduce a class of non-Markovian dynamical systems,
and then postulate the system identification as an optimization problem
using multiple multi-dimensional time-series data.}

\subsection{Linear Dynamical Systems with Homogeneous Memory}

In this section, we consider discrete-time dynamical systems with states
$\bx_t \in \mathbf{R}^n$ and inputs $\bu_t \in \mathbf{R}^k$ for
$t \in \mathbf{Z}_{\ge 0}$.
We introduce a class of linear non-Markovian models with \rev{
\emph{spatially-homogeneous memory} (hereafter referred to as homogeneous memory)}, which forms the modeling basis of the proposed method.

\begin{defn}
\label{defn:Time:invariant:non-Markovian:state-space model}
\textit{(Linear dynamical systems with homogeneous\\ memory):}\quad Assume that $q+1$ initial states $(\bx_t)_{t\in[q+1]}$ and inputs
$(\bu_t)_{t\in[m]}$ with $m \ge q$ are given.
For $t \in [q+1,m+1]$, the state evolution satisfies
\begin{equation}\label{eq:preliminaries:non-Markovian:time:evolution:model}
  YD = AX + BU,
\end{equation}
where $(A,B,D) \in
\mathbf{R}^{n\times n} \times \mathbf{R}^{n\times k} \times \mathbb{T}_m^{q}(Q)$.
Here, 
\begin{align}
X = (\bx_t)_{t\in[m]}^{q}\in \mathbf{R}^{n\times m}, ~
Y = (\bx_{t+1})_{t\in[m]}^{q}\in \mathbf{R}^{n\times m}, \\
U = (\bu_t)_{t\in[m]}^{q}\in \mathbf{R}^{k\times m},
\end{align}
and $D$ has embedding dimension $1 \le Q \le m$.
We call \eqref{eq:preliminaries:non-Markovian:time:evolution:model}
a \emph{linear dynamical system with homogeneous memory}, and refer to $D$
as the \emph{memory kernel matrix}.
\end{defn}
This class of dynamical system is said to have \textit{homogeneous memory} because the memory kernel is spatially homogeneous (independent of the spatial degree of freedom).

\begin{rem}[Constraint]\label{rem:preliminaries:constraint}
  In some cases, we have \textit{a priori} knowledge on the coefficient matrices $A, B, D$ in the linear dynamical systems with homogeneous memory (\ref{eq:preliminaries:non-Markovian:time:evolution:model}), and we can reflect the knowledge to the set of parameters. For example, if we know that all the elements in $A$ are nonnegative, we can impose a \textit{constraint} $A \in \qty(\mathbf{R}_{\geq0})^{n\times n}$. 

  In general, an \textit{a priori} knowledge leads to a constraint of the form $(A,B,D) \in \mathbb{E}$ where $\mathbb{E}$ is a subset of $\mathbf{R}^{n\times n}\times\mathbf{R}^{n\times k}\times \mathbb{T}_m^{q}(Q)$. If the subset $\mathbb{E}$ is a closed convex subset, the analysis of the constraint becomes easy. We list some of the closed convex constraints below:
    \begin{itemize}
        \item Element-wise equality: e.g., $A_{ii} = 0$ for $i\in[n]$ 
        \item Element-wise inequality: e.g., $A_{ii} \geq 0$ for $i\in[n]$ 
        \item Symmetry: e.g., $A=A^\top$ 
    \end{itemize}
  
\end{rem}


\subsection{Linear Systems with Homogeneous Memory and ARX Models}\label{sec:optimization:hankel}

A linear dynamical systems with homogeneous memory (\ref{eq:preliminaries:non-Markovian:time:evolution:model}) is transformed to an autoregressive exogenous (ARX) model. For $t \in[q+1:m]$,
\begin{equation}\label{eq:optimization:ARX:scaler}
  \bx_{t} = A\bx_{t-1} - \sum_{k\in[1:Q]}d_{t-k,t}\bx_{t-k} + B\bu_{t-1},
\end{equation}
where $Q$ is the embedding dimension, $D = \qty(d_{ij})$ and $\bx_t = 0_{n\times1}$ for $t < 0$.

The model can be further transformed into a Hankelized state-space representation
by stacking snapshot vectors to form extended state variables \cite{schmid2010dynamic}.
However, this increases the state dimension from $n$ to $Qn$, leading to coefficient
matrices of size $Qn \times Qn$ and a computational cost that scales as $\mathcal{O}((Qn)^3)$ when
pseudoinverse-based identification is employed.

\subsection{Identification of Dynamical Systems}
\label{sec:optimization:identification}


In large-scale systems, fully exciting all degrees of freedom through
experiments is rarely feasible.
As a result, each experiment typically excites only a subset of the
state--input space, yielding \emph{partially excited} trajectories.

Under partial excitation, system identification becomes ill-posed.
For a single trajectory, the snapshot matrix $X$ and input matrix $U$
generally span a low-dimensional subspace, so that the regression matrix
$(X^\top,U^\top)^\top$ is not full row rank.
This issue is directly reflected in the DMDc regression
\begin{equation}\label{eq:dmdc:single}
  \min_{(A~B)\in\mathbf{R}^{n\times(n+k)}} 
  \left\| Y - \mqty(A~B)\mqty(X\\U) \right\|,
\end{equation}
which admits infinitely many solutions whenever $(X^\top,U^\top)^\top$ is
rank-deficient:
\begin{equation}
  \mqty(A~B)
  = Y\mqty(X\\U)^+
  + M\mqty(\id_{n+k}-\mqty(X\\U)\mqty(X\\U)^+),
\end{equation}
with an arbitrary matrix $M\in\mathbf{R}^{n\times(n+k)}$, which is set to zero in DMDc framework. Only when $(X^\top ~ U^\top)^\top$ is full row rank (corresponding to a fully
excited experiment in the noiseless case) does the solution become unique.

A common approach is to aggregate multiple experiments by stacking the data, $
X=(X_\mu)_{\mu\in[N]},
X'=(X'_\mu)_{\mu\in[N]},
U=(U_\mu)_{\mu\in[N]}$,
which can alleviate rank deficiency caused by individual trajectories, and is equivalent to multi-trajectory DMD \cite{anzaki2024multi} without constraints.
However, in high-dimensional systems, even this multi-trajectory approach may remain
insufficient, and naive stacking alone does not ensure reliable
generalization.

The difficulty is more pronounced for non-Markovian systems.
Although ARX models~\cite{aastrom1971system} and Hankelized representations
(Section~\ref{sec:optimization:hankel}) can represent memory effects, their
parameter dimension grows rapidly with system size or memory length, and
incorporating physical constraints is generally nontrivial.

These observations motivate a modeling framework that
(i) systematically exploits multiple partially excited trajectories and
(ii) directly incorporates physical constraints into the identification problem of non-Markovian dynamical systems, which is the focus of the proposed NOMADS framework.


\subsection{Statement of Problem in NOMADS}


In this subsection, we formulate the system identification problem addressed by NOMADS as a constrained optimization problem.
We consider multiple time-series trajectories obtained from possibly different experimental conditions,
and seek to identify a single global dynamical model that is consistent with all trajectories
while incorporating structural or physical constraints.


Assume we have a set of tuples of time-series data $\mathscr{D}_N$ defined below:
\begin{align}\label{eq:optimization:time:series:data}
  & \mathscr{D}_N \coloneq \left\{\mathscr{D}^{(\mu)}\middle| {\mu\in[N]}\right\},\\
  & \mathscr{D}^{(\mu)} \coloneq \qty(\qty(\bx^{(\mu)}_{t})_{t\in[m_{\mu}]}, \qty(\bu^{(\mu)}_{t})_{t\in[m_{\mu}]}).
\end{align}


Each $\mathscr{D}^{(\mu)}$ is called a \textit{trajectory}, and $N$ denotes the number of trajectories.
The integer $m_{\mu} > 2$ represents the length of the $\mu$-th trajectory.
We fix a positive integer $m < \min_{\mu}(m_{\mu}) - 1$, which determines the number of snapshots used for identification.


We consider the following optimization problem:
\begin{equation}\label{eq:optimization:optimization:problem}
  \begin{split}
  & \mathrm{Minimize} \quad f(A,B,D)\coloneq \sum_{\mu=0}^{N-1}\left\|Y_{\mu}D - \qty(AX_{\mu} + BU_{\mu}) \right\|^2 \\
  & \mathrm{subject~to} \quad (A, B, D) \in \mathrm{St}(A)\times \mathrm{St}(B)\times \mathrm{St}(D),
  \end{split}
\end{equation}
where 
\begin{align}
  X_{\mu} &= \mqty(0_{n\times q} & \bx_{q}^{(\mu)} & \bx_{q+1}^{(\mu)} & \cdots & \bx_{m-1}^{(\mu)}) \in \mathbf{R}^{n\times m}, \\
  Y_{\mu} & =\mqty(0_{n\times q} & \bx_{q+1}^{(\mu)} & \bx_{q+2}^{(\mu)} & \cdots & \bx_{m}^{(\mu)}) \in \mathbf{R}^{n\times m},\\
  U_{\mu} &= \mqty(0_{n\times q} & \bu_{q}^{(\mu)} & \bu_{q+1}^{(\mu)} & \cdots & \bu_{m-1}^{(\mu)}) \in \mathbf{R}^{k\times m},
\end{align}
are data matrices constructed from each trajectory.


The sets $\mathrm{St}(A)\subseteq \mathbf{R}^{n\times n}$,
$\mathrm{St}(B)\subseteq \mathbf{R}^{n\times k}$,
and $\mathrm{St}(D)\subseteq \mathbb{T}_m^{q}$
are closed convex subsets encoding \textit{a priori} knowledge of the system parameters,
as discussed in Remark~\ref{rem:preliminaries:constraint}.
We denote the feasible set by
$\mathbb{E} \coloneq \mathrm{St}(A)\times \mathrm{St}(B)\times \mathrm{St}(D)$. 

An optimal solution to (\ref{eq:optimization:optimization:problem}) is expected to exist in practice,
since $\mathbb{E}$ can be chosen as a bounded set.
For reference, we also define the unconstrained parameter space
$\mathbb{E}_0\coloneq \mathbf{R}^{n\times n}\times \mathbf{R}^{n\times k}\times \mathbb{T}_m^{q}$.

The identification problem~\eqref{eq:optimization:optimization:problem}
is formulated to handle collections of experiments in which each trajectory
may be only partially excited.
The identified model admits an ARX-type representation
via~\eqref{eq:optimization:ARX:scaler}, while keeping the number of free
parameters bounded by $n^2 + nk + m$, which is smaller than that of standard
ARX models or Hankelized state-space representations.

This reduction follows from the assumption of \emph{spatially homogeneous memory}.
While this structural assumption enables scalable identification under partial
excitation, it may limit approximation accuracy when the true memory kernel
exhibits strong spatial heterogeneity.


\begin{rem}
  If $\mathbb{E} = \mathbb{E}_0$, we refer to (\ref{eq:optimization:optimization:problem}) as an \textit{unconstrained} optimization problem.
  If $\mathbb{E} \subsetneq \mathbb{E}_0$, it is called a \textit{constrained} optimization problem.
  The latter corresponds to system identification with structural or physical constraints.
\end{rem}

\begin{rem}
  Several existing DMD-based methods can be recovered as special cases of
  (\ref{eq:optimization:optimization:problem}), as summarized in Table~\ref{tab1}.
  We use the shorthand notation \rev{$X' = \mqty(\bx_{t+1})_{t\in[m]}, X^0 = \mqty(\bx_{t})_{t\in[m]}$.}

  \begin{table}
  \caption{DMD-based methods.}
  \label{tab1}
  \centering
  \setlength{\tabcolsep}{3pt}
  \begin{tabular}{p{40pt}|p{25pt}|p{25pt}|p{25pt}|p{25pt}|p{30pt}|p{30pt}}
  \hline
  Method              & $Y$      & $U$            & $q$      & $\mathrm{St}(A)$              & $\mathrm{St}(B)$         & $\mathrm{St}(D)$\\
  \hline
  DMD \cite{tu2014dynamic}     & \rev{$X'$}   & $0$  & $0$ & $\mathbf{R}^{n\times n}$    & $\{0\}$                   & $\{\id_m\}$\\
  DMDc \cite{proctor2016dynamic}    & \rev{$X'$}   & $U$   & $0$ & $\mathbf{R}^{n\times n}$    & $\mathbf{R}^{n\times k}$  & $\{\id_m\}$\\
  DMDm \cite{anzaki2023dynamic}    & \rev{$X^0$}   & $U$   & $q$ & $\mathbf{R}^{n\times n}$    & $\mathbf{R}^{n\times k}$  & $\mathbb{T}_m^q$\\
  \hline
  \end{tabular}
  \end{table}
\end{rem}


\section{Convex Optimization Scheme}
\label{sec:convex}

This section presents a convex optimization algorithm NOMADS for solving
(\ref{eq:optimization:optimization:problem}).
We show that the loss function is convex and Lipschitz smooth, enabling the use of
projected gradient descent (PGD) with a global sublinear convergence rate.

\subsection{Properties of the Optimization Problem}
\label{sec:scheme:properties}
In this subsection, we analyze fundamental properties of the optimization problem
(\ref{eq:optimization:optimization:problem}).
In particular, we derive explicit expressions for the gradient and Hessian of the loss function,
and establish convexity and smoothness, which form the theoretical basis for the PGD algorithm.
\begin{thm}
  The gradient $\nabla f(A, B, D)$ and the Hessian $\Hess f(A, B, D)$ of the function $f$ are given by the followings:
  \begin{equation}\label{eq:optimization:gradient}
      \nabla f(A, B, D) = \sum_{\mu=0}^{N-1}\qty(-2E_{\mu}X_{\mu}^\top, -2E_{\mu}U_{\mu}^\top, 2Y_{\mu}^\top E_{\mu}),
  \end{equation}
  where $E_{\mu} = Y_{\mu}D - (AX_{\mu} + BU_{\mu})$, and
  \begin{equation}\label{eq:optimization:Hessian}
    \begin{split}
      & \Hess f(A, B, D)[(\Delta_A, \Delta_B, \Delta_D)]\\ &= \sum_{\mu=0}^{N-1}\qty(-2e_{\mu}X_{\mu}^\top, -2e_{\mu}U_{\mu}^\top, 2Y_{\mu}^\top e_{\mu}),
    \end{split}
  \end{equation}
  where $e_{\mu} = Y_{\mu}\Delta_D - (\Delta_AX_{\mu} + \Delta_BU_{\mu})$. Also, the loss function $f$ is a convex function. 
\end{thm}
\begin{proof}
  First we prove the expression for the gradient. For $(\Delta_A, \Delta_B, \Delta_D)\in\mathbf{R}^{n\times n}\times \mathbf{R}^{n\times k}\times \mathbf{R}^{m\times m}$, using the definitions (\ref{eq:preliminaries:inner:prod:matrix}) and (\ref{eq:preliminaries:metric:tuple}),  
  \begin{equation}
    \begin{split}
     &f(A+\Delta_A, B+\Delta_B, D+\Delta_D) - f(A, B, D)\\  =& 2\sum_{\mu
=0}^{N-1}\qty(-\tr\qty(E_{\mu}^\top \Delta_AX_{\mu}) - \tr\qty(E_{\mu}^\top \Delta_BU_{\mu}) + \tr\qty(E_{\mu}^\top Y_{\mu}\Delta_D))\\ & \quad+ \mathcal{O}(\rev{\Delta_A^2+\Delta_B^2+ \Delta_D^2}).
    \end{split}
  \end{equation}
  Use of the identities $\tr(E_{\mu}^\top \Delta_AX_{\mu}) = \tr(X_{\mu}E_{\mu}^\top \Delta_A)$ and $\tr(E_{\mu}^\top \Delta_BU_{\mu}) = \tr(U_{\mu}E_{\mu}^\top \Delta_B)$ yields the expression for the gradient (\ref{eq:optimization:gradient}). 
  
  Second, we prove the expression for Hessian. Noting that the Hessian is the gradient of $\nabla f(A,B,D)$ and
  \begin{equation}
    \begin{split}
      &\nabla f(A+\Delta_A,B+\Delta_B,D+\Delta_D) - \nabla f(A,B,D)\\
      &= \sum_{\mu=0}^{N-1}\qty(-2e_{\mu}X_{\mu}^\top, -2e_{\mu}U_{\mu}^\top, 2Y_{\mu}^\top e_{\mu}),
    \end{split} 
  \end{equation}
  we obtain (\ref{eq:optimization:Hessian}).
  Third, we show the convexity of $f$. From (\ref{eq:preliminaries:inner:prod:matrix}) and (\ref{eq:optimization:Hessian}), 
  \begin{equation}
    \begin{split}
    & \langle \Hess f(A, B, D)[(\Delta_A, \Delta_B, \Delta_D)], (\Delta_A, \Delta_B, \Delta_D)\rangle \\ &= 2\sum_{\mu=0}^{N-1}\|e_{\mu}\|^2 \geq 0.  
    \end{split}.
  \end{equation}
  Therefore, the loss function $f$ is a convex function \cite{roberts1974convex}.
\end{proof}

Note that, the Hessian is a \textit{sum over trajectories} and we can enhance the condition number of the Hessian by adding new trajectories.

In many applications, we are interested in the uniqueness of the coefficient matrices. 
\begin{thm}
  The optimization problem (\ref{eq:optimization:optimization:problem}) has a unique minimizer if $\Hess f(A,B,D)$ is positive definite.
\end{thm}
\begin{proof}
  The function $f$ is a strongly convex function, because Hess $f(A,B,D)$ is positive definite.
  Thus, the sub-level set of $f$ is a bounded subset. 
  This implies that the optimization problem (\ref{eq:optimization:optimization:problem}) has a unique minimizer \cite{bertsekas2016nonlinear}.
\end{proof}

In general, the minimizer of (\ref{eq:optimization:optimization:problem}) is not unique, because the Hessian (\ref{eq:optimization:Hessian}) has nonempty kernel. 

\ifarxiv
\begin{rem}
  In general, the Hessian  (\ref{eq:optimization:Hessian}) cannot be expressed as a matrix, because it contains linear maps $\mathbf{R}^{a\times b}\to\mathbf{R}^{c\times d}$ of the form \rev{$\sigma_{P,Q}:M \mapsto P M Q$ for given matrices $P\in\mathbf{R}^{c\times a}, Q\in\mathbf{R}^{b\times d}$ and for any matrix $M\in\mathbf{R}^{a\times b}$}. 
On the other hand, when the set for the memory kernel is a singleton,
i.e., $\mathrm{St}(D)=\{D\}$, the Hessian admits a matrix representation.
In this case, the Hessian with respect to $(A,B)$ is given by
\begin{equation}\label{eq:optimization:hessian:mtdmd}
  H(A,B)
  = \sum_{\mu=0}^{N-1}
  \mqty(
    X_{\mu}X_{\mu}^\top & X_{\mu}U_{\mu}^\top \\
    U_{\mu}X_{\mu}^\top & U_{\mu}U_{\mu}^\top
  ),
\end{equation}
which coincides with the Hessian obtained in~\cite{anzaki2024multi}.
\end{rem}
\fi

\subsection{Projected Gradient Descent Method}

We can solve the optimization problem (\ref{eq:optimization:optimization:problem}) using PGD method shown in Algorithm \ref{alg:NOMADS}. We call Algorithm \ref{alg:NOMADS} NOMADS, standing for \textit{Non-Markovian Optimization-based Modeling for Approximate Dynamics with Spatially-homogeneous memory}.

\begin{algorithm}[t]
  \caption{Algorithm for NOMADS}\label{alg:NOMADS}
  \begin{algorithmic}[1]
  \renewcommand{\algorithmicrequire}{\textbf{Require:}}
  \renewcommand{\algorithmicensure}{\textbf{Output:}}
  \REQUIRE initial stepsize $t_0>0$, constant $\eta > 1$, initial guess $\theta^{(0)}\in\mathbb{E}_0$, convex projection $\mathcal{P}_{\mathbb{E}}$ onto $\mathbb{E}\subset\mathbb{E}_0$, maximum number of steps $p\geq1$ 
   \FOR {$\ell = 0$ to $p-1$}
   \STATE $\theta' \leftarrow \mathcal{P}_{\mathbb{E}}\qty(\theta^{(\ell)} - t_{\ell}\nabla f(\theta^{(\ell)}))$
   \STATE $f' \leftarrow f(\theta^{(\ell)}) + \left\langle (\theta' - \theta^{(\ell)}), \nabla f(\theta^{(\ell)})\right\rangle + (2t_{\ell})^{-1}\left\|\theta' - \theta^{(\ell)}\right\|^2$
   \WHILE {$f(\theta') > f'$}
   \STATE $t_{\ell} \leftarrow t_{\ell} / \eta$
   \STATE $\theta' \leftarrow \mathcal{P}_{\mathbb{E}}\qty(\theta^{(\ell)} - t_{\ell}\nabla f(\theta^{(\ell)}))$
   \STATE $f' \leftarrow f(\theta^{(\ell)}) + \left\langle (\theta' - \theta^{(\ell)}), \nabla f(\theta^{(\ell)})\right\rangle + (2t_{\ell})^{-1}\left\|\theta' - \theta^{(\ell)}\right\|^2$
   \ENDWHILE
   \STATE $\theta^{(\ell+1)} \leftarrow \theta'$
   \STATE $t_{\ell+1} \leftarrow t_{\ell}$
   \ENDFOR
  \RETURN $\theta^{(p)}$
  \end{algorithmic}
\end{algorithm}
In Algorithm \ref{alg:NOMADS}, the convex projection $\mathcal{P}_{\mathbb{E}}$ onto $\mathbb{E}$ is defined as 
\begin{equation}\label{eq:optimization:convex:projection}
  \mathcal{P}_{\mathbb{E}}:\theta \mapsto \mathrm{arg}\min_{\vartheta\in\mathbb{E}}\|\theta-\vartheta\|.
\end{equation}
The exact form of the projection operator $\mathcal{P}_{\mathbb{E}}$ depends on the space $\mathbb{E}$ \cite{beck2017first}. Using (\ref{eq:preliminaries:metric:tuple}) and the fact that $\mathrm{St}(A), \mathrm{St}(B)$, and $\mathrm{St}(D)$ are closed convex sets, we can see that for the convex projection (\ref{eq:optimization:convex:projection})  there exists a unique tuple of convex projections $\mathcal{P}_{\mathrm{St}(A)}:\mathbf{R}^{n\times n}\to\mathrm{St}(A)$, $\mathcal{P}_{\mathrm{St}(B)}:\mathbf{R}^{n\times k}\to\mathrm{St}(B)$, and $\mathcal{P}_{\mathrm{St}(D)}:\mathbf{R}^{m\times m}\to\mathrm{St}(D)$ that satisfies:
\begin{equation}\label{eq:optimization:convex:projection:component}
  \mathcal{P}_{\mathbb{E}}(A,B,D) = \qty(\mathcal{P}_{\mathrm{St}(A)}(A), \mathcal{P}_{\mathrm{St}(B)}(B),
  \mathcal{P}_{\mathrm{St}(D)}(D)).
\end{equation}
For a matrix $D\in\mathbf{R}^{m\times m}$, the convex projection $\mathcal{P}_{\mathbb{T}_m^q(Q)}(D) = \qty(E_{ij})\in\mathbf{R}^{m\times m}$ is expressed as:
\begin{equation}
  E_{ij} = 
    \begin{cases}
      0 & i > j \\
      \qty(m_{ij}-q_{ij})^{-1}\sum_{k=q_{ij}}^{m_{ij}-1}D_{k, k+(j-i)} & i\leq j < i+Q\\
      0 & j\geq i+Q
    \end{cases},
\end{equation}
where $m_{ij} = m-(j-i)$ and $q_{ij} = \max(0, q-(j-i))$.
Note that the projection $\mathcal{P}_{\mathbb{T}_m^q(Q)}$
can be defined due to Theorem \ref{Thm5}.

The function $f$ is Lipschitz smooth as shown in the following theorem, which is used in the proof of Theorem \ref{thm:optimization:Convergence}. 
\begin{thm}[Lipschitz smoothness of $f$]\label{thm:optimization:Lipschitz}
  In the optimization problem (\ref{eq:optimization:optimization:problem}), the function $f$ is $L_{f}$-smooth on $\mathbb{E}_0$, where 
  \begin{equation}\label{eq:optimization:Lipschitz:constant}
    L_{f} \coloneq 2\sum_{\mu=1}^N\qty(\|X_{\mu}\|^2 + \|U_{\mu}\|^2 + \|Y_{\mu}\|^2).  
  \end{equation}
\end{thm}
\begin{proof}
Since the Hessian operator (\ref{eq:optimization:Hessian}) is a constant operator $\Hess f$ independent of $(A, B, D)$,
  \begin{equation}
    \|\nabla f(\theta') - \nabla f(\theta)\| \leq \|\Hess f\| \|\theta - \theta'\|,
  \end{equation}
  for any $\theta = (\Delta_A, \Delta_B, \Delta_D), \theta' \in \mathbb{E}_0$. Here $\|\mathrm{Hess}\,f\|$ denotes the operator norm induced by the Frobenius norm. 

For nonzero $(\Delta_A, \Delta_B, \Delta_D)$, 
\begin{align}\nonumber
&\frac{\|\Hess f[(\Delta_A, \Delta_B, \Delta_D)]\|}{\|(\Delta_A, \Delta_B, \Delta_D)\|}\\ \nonumber
 &= \sqrt{\frac{4\sum_{\mu=1}^N\|e_{\mu}\|^2\qty(\|X_{\mu}\|^2 + \|U_{\mu}\|^2 + \|Y_{\mu}\|^2)}{\|\qty(\Delta_A, \Delta_B, \Delta_D)\|^2}}\\
&\le  2\sum_{\mu=1}^N\qty(\|X_{\mu}\|^2 + \|U_{\mu}\|^2 + \|Y_{\mu}\|^2),
\end{align}
where we used 
\begin{align}\nonumber
\|e_{\mu}\|^2 &= \qty(\|\Delta_A\|^2\|X_{\mu}\|^2+ \|\Delta_B\|^2\|U_{\mu}\|^2+ \|\Delta_D\|^2\|Y_{\mu}\|^2)\\
& \le \|\qty(\Delta_A, \Delta_B, \Delta_D)\|^2\qty(\|X_{\mu}\|^2 + \|U_{\mu}\|^2 + \|Y_{\mu}\|^2)
\end{align}
Therefore $\|\Hess f\| \le L_f$, meaning,
  \begin{equation}
    \|\nabla f(\theta') - \nabla f(\theta)\| \leq L_{f}\|\theta - \theta'\|,
  \end{equation}
  for any $\theta, \theta' \in \mathbb{E}_0$.
\end{proof}
Now we show the convergence of Algorithm \ref{alg:NOMADS}. 
\begin{thm}[Convergence of the PGD algorithm]\label{thm:optimization:Convergence}
  Assume that (\ref{eq:optimization:optimization:problem}) has an optimal solution. If $p=\infty$, Algorithm \ref{alg:NOMADS} generates a sequence of parameters $\theta_{\ell}\in\mathbb{E} ~ (\ell\in\mathbf{Z}_{\geq\rev{1}})$ that converges to an optimal solution of problem (\ref{eq:optimization:optimization:problem}).
\end{thm}
\begin{proof}
  From Theorem \ref{thm:optimization:Lipschitz}, the function $f$ is $L_f$-smooth for (\ref{eq:optimization:Lipschitz:constant}). 
  \rev{We can also show that $f(\theta) \ge 0$ for any $\theta \in \mathbb{E}$, and that there exists a $\theta \in \mathbb{E}$ such that $f(\theta) < \infty$ provided that the data matrices $X_\mu, Y_\mu, U_\mu$ and the number of trajectories $N$ are finite.}
  Using Theorem 10.24 in \cite{beck2017first}, we can show the convergence of Algorithm \ref{alg:NOMADS}.
\end{proof}

 We can remove the assumption that (\ref{eq:optimization:optimization:problem}) has an optimal solution, if
the set of parameters $\mathbb{E}$ in the optimization problem (\ref{eq:optimization:optimization:problem})  is bounded.

\begin{rem}
    In case the constrained parameter space $\mathbb{E}$ is not a convex subset of $\mathbb{E}$, Algorithm \ref{alg:NOMADS} will generate a sequence of parameters $\theta_{\ell}\in\mathbb{E} ~ (\ell\in\mathbf{Z}_{\geq1})$ that converges to a stationary point of the objective function $f$.
\end{rem}

\begin{rem}
  Although the convergence of NOMADS is assured even if we use a constant stepsize $1/L_f\geq t>0$ \cite{beck2017first}, we expect a better behavior by \textit{adjusting} stepsize as shown in Algorithm \ref{alg:NOMADS} because the Lipschitz constant $L_f$ is often a huge real number depending on the data (\ref{eq:optimization:time:series:data}).
\end{rem}

\begin{thm}[Convergence rate of NOMADS]\label{thm:optimization:Convergence:rate}
Assume that problem~\eqref{eq:optimization:optimization:problem}
has an optimal solution.
Then Algorithm~\ref{alg:NOMADS} converges at a sublinear rate $O(1/\ell)$.
\end{thm}

\begin{proof}
By Theorem~\ref{thm:optimization:Lipschitz}, the objective function $f$ is convex
with a Lipschitz continuous gradient, and the feasible set
$\mathbb{E}=\mathrm{St}(A)\times\mathrm{St}(B)\times\mathrm{St}(D)$
is closed and convex.
Thus, Algorithm~\ref{alg:NOMADS} is an instance of projected gradient descent
for a smooth convex optimization problem.
The stated convergence rate follows from
\cite{nesterov2004introductory}.
\end{proof}


\section{Numerical Experiments}
\label{sec:numerical}

We conduct numerical experiments on synthetic data to evaluate NOMADS,
focusing on \emph{generalization performance} under \emph{partially excited} experiments,
motivated by multi-zone thermostat dynamics with heat exchange with the environment.
Models are trained on multiple synthetic data and evaluated on unseen test data,
reflecting the challenges discussed in
Section~\ref{sec:optimization:identification}.

\subsection{Experimental Set-ups}

We consider a synthetic benchmark defined on a two-dimensional grid
$[L_x]\times[L_y]$ with periodic boundary conditions in the first direction,
corresponding to a cylindrical topology.
Let $\mathcal{L}\in\mathbf{R}^{L_xL_y\times L_xL_y}$ denote a graph Laplacian on
this grid, where coupling strengths between adjacent cells are drawn
independently from a uniform distribution on the interval $[w_0,w_1]\subset \mathbf{R}$.
We first consider the following Markovian linear dynamical system:
\begin{equation}\label{eq:numerical:Experiment:Markov}
  \bx_{t+1} = \mathcal{A}\bx_t + \mathcal{B}\bu_t, \quad \bx_0 = 0_{L_xL_y\times1},
\end{equation}
where $\mathcal{A} := \id_{L_xL_y} + \mathcal{L}h$ and
$\mathcal{B} := \id_{L_xL_y}h$, with time step $h>0$.
The input sequence $\{\bu_t\}_{t\geq0}$ represents external actuation.

To model memory effects arising from system--reservoir interactions, we also consider a non-Markovian extension.
Specifically, we set $A=\mathcal{A}$, $B=\mathcal{B}$, and introduce a memory kernel matrix $D=\mathcal{D}$ in Definition~\ref{defn:Time:invariant:non-Markovian:state-space model},
with given integers $m>q\geq0$ and $\mathcal{D}\in\mathbb{T}_m^{q}(Q)$.
Such memory effects naturally arise in systems coupled to unobserved degrees of freedom such as external reservoirs \cite{svenkeson2016spectral}.

Throughout the experiments, we fix
$L_x = 20$, $L_y=5$, $n = L_xL_y = 100$, $k=n$, $m = 1,000$, $h = 1/(m-1)$, $w_{0} = 0.5$, $w_{1} = 1.5$, $q=2$, $Q=3$, and 
\begin{equation}
  \mathcal{D} = \mqty(0 & 0 & -0.01 & 0 & \cdots\\
                      0 & 0 & 0.03  & -0.01 & \\
                      0 & 0 & 1     & 0.03  & \\
                      \vdots &&&&)\in\mathbb{T}_m^q(Q).
\end{equation}

Define a matrix $U_{\sigma, \nu, \xi}^{\bullet} = \qty(u_{it}) \in \mathbf{R}^{n\times m}$ as follows: for $t = 0,1,2,\dots$, spatial coordinates $x,y$, and $i\in[n]$, 
\begin{align}
  u_{it} = \begin{cases}\sigma^{x}\delta_{\xi,y}\sin(2\pi \nu th) & \bullet = \parallel\\ \sigma^{y}\delta_{\xi,x}\sin(2\pi \nu th) & \bullet = \perp\end{cases}.
\end{align}
Note that, in the Markovian case with zero control input, the \emph{energy} is conserved.
We define the energy at time $t$ as
\begin{equation}\label{eq:num:energy}
  E(t) = \frac{1}{n}\sum_{i\in[n]} (\bx_t)_i .
\end{equation}

\subsection{Train and Test Datasets}
We generate train, test and energy test sets by using the following sets of pairs of control inputs and initial values $\mqty(U, \mqty(\bx_0 & \bx_1 & \cdots & \bx_q))$ as follows:
\begin{align}\label{eq:numerical:Experiment:train:u}
  \mathcal{S}_{\mathrm{tr}} &= \left\{\qty(U^{\parallel}_{\sigma, \nu, \xi}, 0_{n\times q}) \middle| \sigma =\pm1, \nu\in\{3,6\}, \xi\in[5]\right\},\\ \label{eq:numerical:Experiment:test:u}
  \mathcal{S}_{\mathrm{ts}} &= \left\{\qty(U^{\perp}_{1, \nu, 0}, 0_{n\times q}) \middle| \nu\in[1,9]\right\},\\ \label{eq:numerical:Experiment:energy:test:u}
  \mathcal{S}_{\mathrm{en}} &= \left\{\qty(0_{n\times m}, 1_{n\times q}) \right\},
\end{align}
respectively. We assume $q = 0$ for the Markovian case. 

We prepare pairs of train and test datasets both for Markovian (\ref{eq:numerical:Experiment:Markov}) and non-Markovian (Definition \ref{defn:Time:invariant:non-Markovian:state-space model}) dynamical systems using (\ref{eq:numerical:Experiment:train:u}) and (\ref{eq:numerical:Experiment:test:u}). \rev{
Noisy train data is generated using three independent noise realizations.
For each training trajectory $\tau \in \mathcal{S}_{\mathrm{tr}}$, the root mean square of the snapshot matrix
$X^{(\tau)}$ is defined as $\rho_\tau := \|X^{(\tau)}\| / \sqrt{nm}$, and additive Gaussian noise with
standard deviation $\sigma = 0.05 \times \max_{\tau\in\mathcal{S}_{\mathrm{tr}}} \rho_\tau$
is applied to the state observations, while no noise is added to the test datasets
$\mathcal{S}_{\mathrm{ts}}$ and $\mathcal{S}_{\mathrm{en}}$.
}

\subsection{Constraints and Learning}
We use following constraints A1 or A2 and B to identify the coefficient matrices $A, B$:
\begin{itemize}
  \item[A1] $A^\top=A$, $A_{ij}=0$ if $j\not\in N_i$, and off-diagonal elements of $A$ are nonnegative
  \item[A2] $A$ is a graph Laplacian, $A_{ij}=0$ if $j\not\in N_i$, and off-diagonal part of $A$ is nonnegative
  \item[B] $B$ is a diagonal matrix whose elements are nonnegative
\end{itemize}
For A2, we used Algorithm 4 in \cite{sato2024nearest} to construct the convex projection onto the set of (asymmetric) graph Laplacians.

In the learning procedure, we set $t_0 = 0.3$, $\eta = 1.05$, $\theta^{(0)} = (\id_{n}, 0_{n\times k}, \id_m)$, and $p=10,000$ in Algorithm \ref{alg:NOMADS}.

\subsection{Modeling using DMDc and DMDm}
\label{sec:num:modeling:using:dmdc}


We use DMDc~\cite{proctor2016dynamic} and DMDm~\cite{anzaki2023dynamic}
as baseline methods.
For both methods, no rank truncation or dimensionality reduction is applied,
since truncation based on train data would bias the identified models
toward the excited subspace and impair generalization under partial excitation.

DMDc is applied under the Markovian assumption disregarding the memory effect using a multi-trajectory formulation
(Section~\ref{sec:optimization:identification}),
while DMDm does not naturally admit multi-trajectory identification and is therefore
fit independently to each training trajectory, with a single representative model
constructed by element-wise averaging of the estimated coefficients.


\subsection{Results}
\label{sec:num:results}


This section reports numerical results evaluating NOMADS in terms of
prediction accuracy, physical consistency, and optimization behavior
under partially excited experimental settings.

\subsubsection{Overview of computational cost and prediction accuracy}


Table~\ref{tab:cost_accuracy} reports modeling time (average over 3 trials for noisy data), prediction RMSE (root mean squared error),
and, for the Markovian case, the maximum absolute energy error.
Across both Markovian and non-Markovian settings,
NOMADS achieves substantially smaller test reconstruction errors than
DMD-based baselines.

The symmetry constraint~(A1) generally yields lower RMSE,
whereas the graph Laplacian constraint~(A2) more accurately preserves
energy in the Markovian case.


\begin{table*}[ht]
\centering

\caption{Computational cost and prediction accuracy.}
\label{tab:cost_accuracy}
\begin{tabular}{l c c c c c c }
\toprule
& \multicolumn{1}{c}{} & \multicolumn{3}{c}{Markov} & \multicolumn{2}{c}{Non-Markov} \\
\cmidrule(lr){3-5}\cmidrule(lr){6-7}
Method & \makecell{Modeling \\ Time [s]} & \makecell{RMSE \\(noiseless)} & \makecell{RMSE \\(noisy)} & \makecell{Max absolute \\ energy error}  & \makecell{RMSE \\(noiseless)} & \makecell{RMSE \\(noisy)} \\
\midrule
DMDc \cite{proctor2016dynamic}          & $\mathbf{0.89}$ & 0.33   &      $(2.20 \pm 2.01)\times 10^{12}$        & $5.01\times10^{-6}$    & $5.85\times10^{104}$ & $(7.12 \pm 4.62)\times 10^{11}$ \\
DMDm \cite{anzaki2023dynamic}          & 5.45 & $3.11\times10^{17}$ &  $(5.07 \pm 7.14)\times 10^{10}$ & $4.33\times10^{19}$    & NaN & NaN \\
NOMADS (A1+B)  & $2.99\times10^{4}$ & $\mathbf{1.06\times10^{-2}}$ & $ \mathbf{(1.12 \pm 0.0015)\times 10^{-1}}$ & $3.20\times10^{-3}$  & $\mathbf{1.02\times10^{-2}}$ & $\mathbf{(2.86 \pm 0.025)\times 10^{-2}}$\\
NOMADS (A2+B)  & $2.87\times10^{4}$ & $5.28\times10^{-2}$ & $(1.43 \pm 0.027)\times 10^{-1}$ & $\mathbf{1.57\times10^{-13}}$  & $2.92\times10^{-2}$ & $(3.51\pm0.046)\times10^{-2}$ \\
\bottomrule
\end{tabular}

\end{table*}


\subsubsection{Reconstruction results}

\begin{figure}[ht]
    \centering
    \begin{minipage}[t]{0.49\columnwidth}
      \centering
      \includegraphics[width=\linewidth,viewport=0 0 216 432,
  clip]{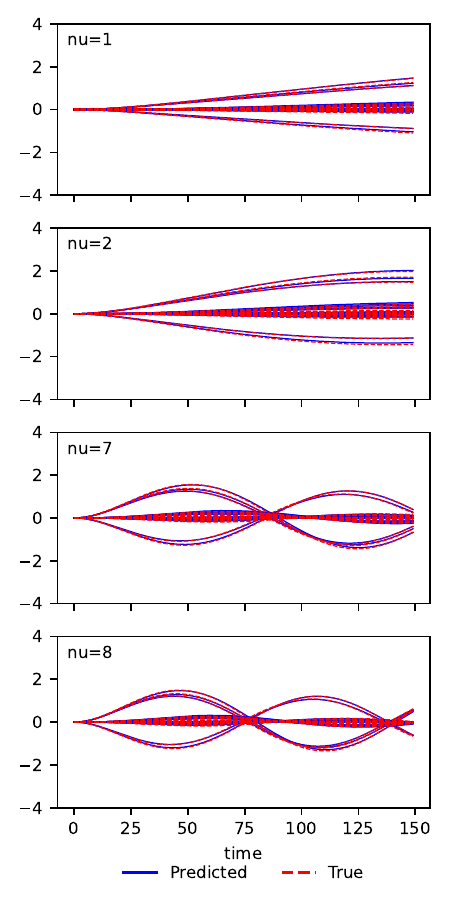}
    \end{minipage}\hfill
    \begin{minipage}[t]{0.49\columnwidth}
      \centering
      \includegraphics[width=\linewidth,viewport=0 0 216 432,
  clip]{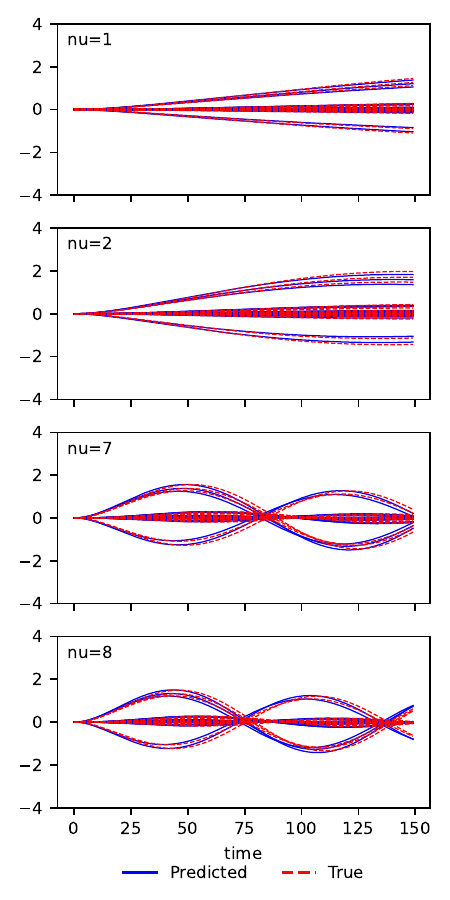}
    \end{minipage}
  \caption{\rev{NOMADS with A2 and B constraints for the non-Markovian dynamical system.
Comparison between the ground truth (solid lines) and reconstructed trajectories (dotted lines),
evaluated on $\mathcal{S}_{\mathrm{ts}}$.
Left column: models trained on noiseless training data.
Right column: models trained on noisy training data.}}
  \label{fig:num:markov:short:nomads} 
\end{figure}
\begin{figure}
      \centering
    \begin{minipage}[t]{0.49\columnwidth}
      \centering
      \includegraphics[width=\linewidth,viewport=0 0 216 432,
  clip]{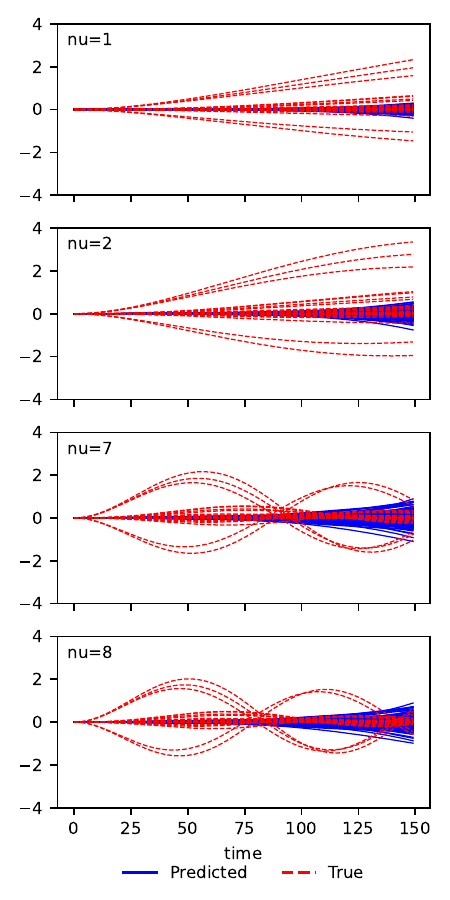}
    \end{minipage}\hfill
    \begin{minipage}[t]{0.49\columnwidth}
      \centering
      \includegraphics[width=\linewidth,viewport=0 0 216 432,
  clip]{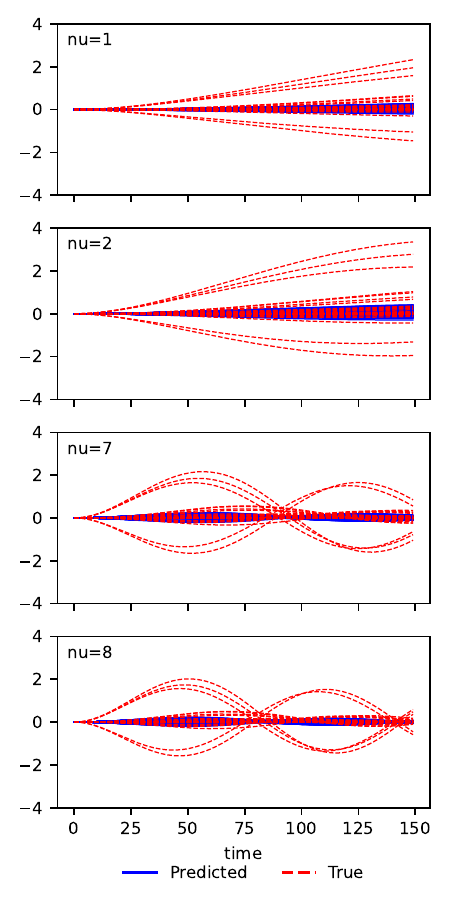}
    \end{minipage}
  \caption{
  Existing methods. Ground truth (solid) and reconstructed trajectories (dotted) for the Markovian dynamical system. Models are identified from noiseless train data $\mathcal{S}_{\mathrm{tr}}$ and evaluated on $\mathcal{S}_{\mathrm{ts}}$. Left: DMDm. Right: DMDc trained on horizontally concatenated trajectories.
  }
  \label{fig:num:markov:short:dmd}  
\end{figure}


Figures~\ref{fig:num:markov:short:nomads} and~\ref{fig:num:markov:short:dmd} compare reconstructed trajectories with the ground truth for both Markovian and non-Markovian systems.
NOMADS accurately reproduces the long-term dynamics in both noiseless and noisy training settings, whereas DMD-based methods exhibit significant deviations, particularly for non-Markovian dynamics, even when trained on noiseless data.


\subsubsection{Energy conservation}


\begin{figure}[ht]
  \centering
  \includegraphics[width = \linewidth, viewport=10 10 560 270, clip]{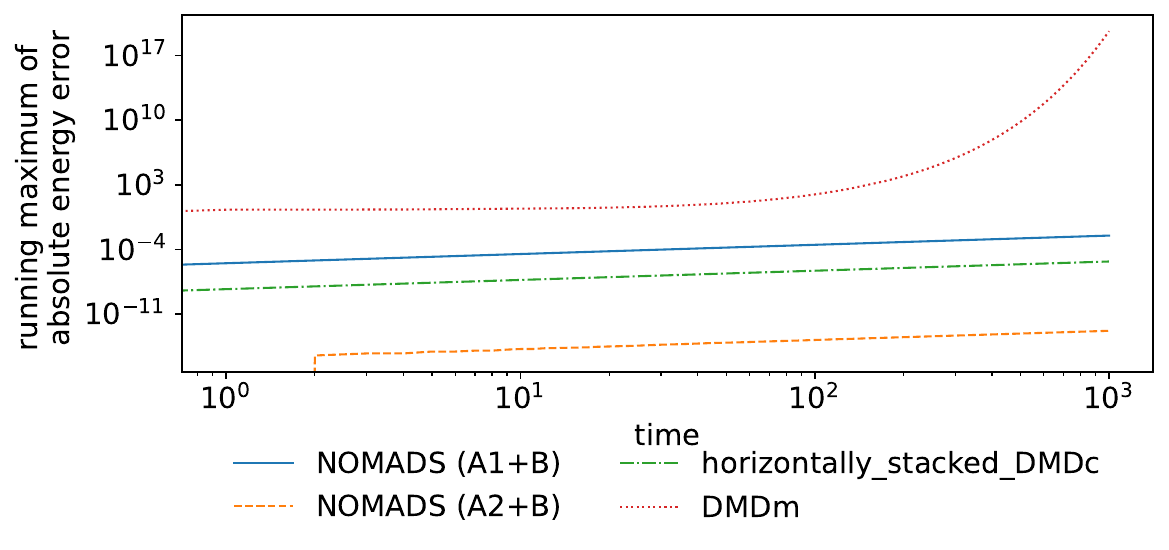}
  \caption{
Running maximum of the absolute energy error in the Markovian settings, where the energy \(E(t)\) is defined in \eqref{eq:num:energy}.
The initial condition is set to $\bx(0)=\mathbf{1}_{n\times 1}$. All models are trained on noiseless training data.
}
  \label{fig:num:energy}
\end{figure}


Figure~\ref{fig:num:energy} shows the time evolution of the energy error
in the Markovian setting, where energy conservation holds. The models are trained on noiseless data.
NOMADS with the graph Laplacian constraint~(A2) achieves the smallest
energy error.

\subsubsection{Learning curves}

\begin{figure}[ht]
  \raggedright
  \hspace{8mm}  \includegraphics[
  width = \linewidth, viewport=20 0 530 200, clip]{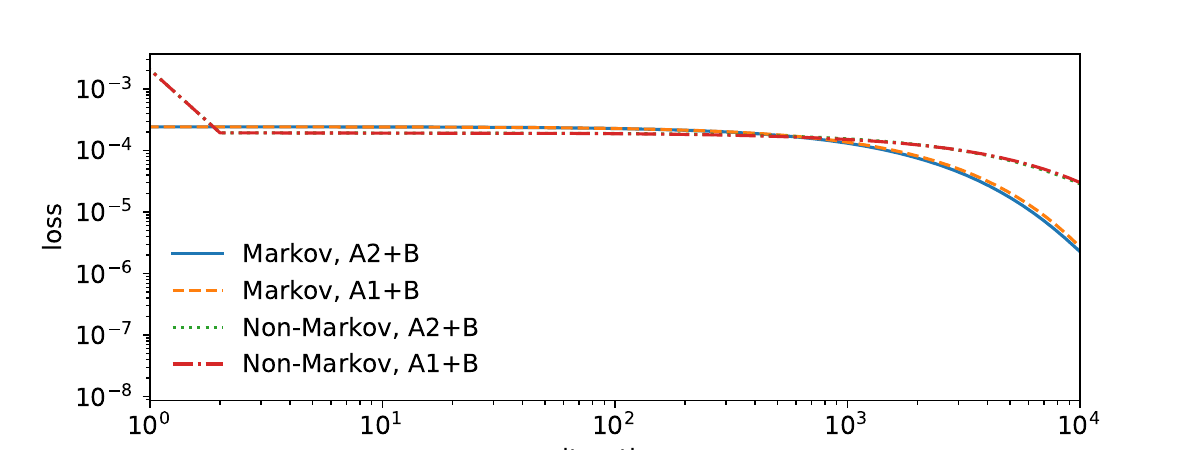}
  \caption{\rev{Learning curves of the training loss for NOMADS under different constraints.
All models are trained on noiseless training data.}}
  \label{fig:learning:curve}
\end{figure}


Figure~\ref{fig:learning:curve} shows the learning curves of NOMADS.
In the non-Markovian cases, the loss decreases rapidly in early iterations
and then converges gradually, reflecting the effect of projection onto
the feasible set of causal memory kernels.


\subsection{Discussion}
This section discusses the numerical results in light of the three challenges
identified in the introduction:
(i) generalization under partial excitation,
(ii) physical consistency through constraints,
and (iii) modeling of non-Markovian dynamics.
As emphasized earlier, the first two challenges are closely related and are
addressed through multi-trajectory learning with physical constraints,
while the third concerns the expressiveness of the dynamical model itself.

First, regarding reconstruction accuracy under partial excitation,
NOMADS consistently achieves smaller test errors than DMD-based methods
in both Markovian and non-Markovian settings.
Since the test trajectories differ qualitatively from the train data,
this improvement indicates genuine generalization rather than overfitting.
In contrast, DMDc and DMDm suffer from large errors, reflecting the
non-identifiability issues discussed in
Section~\ref{sec:optimization:identification}.

Second, in the Markovian case, NOMADS with the graph Laplacian constraint improves energy conservation,
while that with symmetry constraints tend to reduce reconstruction error.
This highlights a trade-off between numerical accuracy and strict enforcement
of physical structure, depending on the constraint design.

Third, NOMADS remains effective in the non-Markovian setting with noisy train data,
where standard DMD-based approaches fail to capture memory effects.
By explicitly incorporating a memory kernel into the optimization framework,
NOMADS yields stable long-term predictions even when individual trajectories
are only partially excited.

Finally, in the present setting, the underlying physical dynamics suggests that a symmetric graph Laplacian constraint would be the most natural choice.
Efficient projection onto the set of symmetric graph Laplacians remains computationally challenging.
While projection-based methods such as those of Sato \textit{et al.}~\cite{sato2024nearest} are principled, their repeated use within an iterative optimization framework like NOMADS would incur substantial computational overhead, and is therefore left for future work.


\section{Concluding Remarks}
\label{sec:concluding}

In this paper, we proposed NOMADS (Non-Markovian Optimization-based Modeling for
Approximate Dynamical Systems with Spatially-homogeneous memory), a system identification framework for dynamical systems using collections of partially excited experiments.
NOMADS integrates multiple trajectories, incorporates homogeneous memory effects,
and enforces physically motivated constraints within a convex optimization
formulation.

We established convergence guarantees for the projected gradient-based algorithm
used in NOMADS and demonstrated its effectiveness through numerical experiments using synthetic data.
The results show that NOMADS achieves lower reconstruction errors on test data
than DMD-based methods even for non-Markovian systems with noisy train data, indicating improved generalization performance measured
in terms of test-set reconstruction accuracy.
In our experiments, the training and test experiments are intentionally
qualitatively different, highlighting the robustness of the proposed framework.

Moreover, incorporating physically motivated constraints contributes both to
improved reconstruction accuracy and to the preservation of physical properties
such as conservation laws.

These results indicate that NOMADS provides a reliable and flexible approach for
identifying high-dimensional non-Markovian dynamical systems under partial
excitation.

\appendices

\section*{Acknowledgment}

The authors thank Mr. Ren Sasaki (Equipment Intelligence \& App R\&D Department, Tokyo Electron Ltd.) for his assistance with the numerical computations of DMDm. This work was partly supported by Japan Society for the Promotion of Science KAKENHI under Grant 23K28369.


\ifarxiv
  \section{Linear dynamical systems with homogeneous memory and Fractional Dynamics}
    The time-evolution model in fractional DMD \cite{anzaki2023dynamic} is a fractional-order ordinary differential equation with a positive order $\alpha>0$:
    \begin{equation}
      \dv[\alpha]{\bx}{t} = A\bx(t) + B\bu(t).
    \end{equation}
    Note that the fractional derivative $\dv[\alpha]{}{t}$ is represented as the composition of integer-order and fractional derivatives as $\dv[\alpha]{}{t} = \dv[n_{\alpha}]{}{t}\dv[\alpha-n_{\alpha}]{}{t}$ with $n_{\alpha}$ being the minimum positive integer such that $n_{\alpha} -\alpha > 0$. The kernel of the original operator $\dv[\alpha]{}{t}$ is equivalent to the kernel of the integer-order derivative $\dv[n_{\alpha}]{}{t}$, i.e., 
    \begin{eqnarray}
      \ker\qty(\dv[\alpha]{}{t}) = \mathrm{span}\left\{t^{\ell}\middle| \ell\in[n_{\alpha}]\right\}.
    \end{eqnarray}
    Therefore, we multiply the pseudoinverse operator $\dv[-\alpha]{}{t}$ from left to the both sides of the above equation to yield
    \begin{equation}\label{eq:fracdmd}
      \bx(t) - \sum_{\ell = 0}^{n_{\alpha}}\bx_0^{(\ell)}t^\ell = \dv[-\alpha]{(A\bx + B\bu)}{t},
    \end{equation}
    where $\bx_0^{(\ell)} \in \mathbf{R}^n$ is the $\ell$th derivative of $\bx(t)$ at $t=0$. 
    
    In the discrete time representation for given time points $(t_i)_{i\in[m]}$ with equal spacing $\Delta t \coloneq t_{i+1} - t_i$, the fractional-order derivative is represented by a $m$-dimensional time-invariant causal matrix $D_{\alpha}\in\mathbf{R}^{m\times m}$ which satisfies the following conditions:
    \begin{enumerate}
      \item $D_{0} = \id_m$,
      \item $\mqty(0^b ~ 1^b ~ \dots ~ (m-1)^b)D_{a} = 0_{n\times m}$ for integers $b < a$,
      \item $D_{\alpha}D_{\beta} = D_{\alpha+\beta}$ for $\alpha, \beta > 0$.
    \end{enumerate}
  We can show that $\ker(D_{\alpha}^\top) = \ker(D_{n_{\alpha}}^\top) = \mathrm{span}\left\{\qty[v_i]\in\mathbf{R}^{m}\middle| v_i = (t_i)^\ell, ~ i\in[m], ~ \ell\in[n_{\alpha}]\right\}$. Noting that the discretization induces $\dv[-\alpha]{}{t}\to D_{\alpha}^+$, we can see that this is a special case of general memory kernel matrix.
  
  \section{Uniqueness in DMD, DMDc, and NOMADS}

  In this section, we focus on the uniqueness of the NOMADS, MTDMD, and DMDc. There are limited attentions \cite{hirsh2020centering} on the uniqueness, but in our \textit{from-analysis-to-identification} perspective, the uniqueness of the coefficient is of utmost interest.

  \begin{thm}[Uniqueness in NOMADS]\label{thm:optimization:unique}
    Under the settings of Theorem \ref{thm:optimization:Convergence}, assume that $\mathrm{St}(A) = \mathbf{R}^{n\times n}$, $\mathrm{St}(B) = \mathbf{R}^{n\times k}$, $\mathrm{St}(D) = \mathbb{T}_m^q$. The optimization problem (\ref{eq:optimization:optimization:problem}) has a unique minimizer if and only if the Hessian (\ref{eq:optimization:Hessian}) is a positive definite operator.
    \end{thm}

  If $D$ is fixed, the theorem becomes simpler:
  \begin{thm}[Uniqueness in Multi-trajectory DMDm]\label{thm:optimization:unique:multi:DMDm}
    Under the settings of Theorem \ref{thm:optimization:Convergence}, further assume that $\mathrm{St}(A) = \mathbf{R}^{n\times n}$, $\mathrm{St}(B) = \mathbf{R}^{n\times k}$, $\mathrm{St}(D) = \{D\}$ for a given matrix $D\in\mathbf{R}^{m\times m}$. The optimization problem (\ref{eq:optimization:optimization:problem}) has a unique minimizer if and only if the Hessian (\ref{eq:optimization:hessian:mtdmd}) is a positive definite matrix.
    \end{thm}

  \begin{cor}[Uniqueness in DMDc]\label{rem:optimization:unique:DMDc}
    Under the settings of Theorem \ref{thm:optimization:unique:multi:DMDm}, let us further assume $N=1$ and $D = \id_m$. This setting is corresponding to DMDc. The optimization problem (\ref{eq:optimization:optimization:problem}) has a unique minimizer if and only if 
    \begin{equation}
      \rank\mqty(X \\ U) = n+k.
    \end{equation}
    Note that, the above equality is satisfied only if $m\geq n+k$.
  \end{cor}
  \begin{proof}
    In this case the Hessian becomes a $(n+k)$-dimensional square matrix (\ref{eq:optimization:hessian:mtdmd}):
    \begin{equation}
      \begin{split}
        \Hess f(A,B) &= \mqty(XX^\top & XU^\top \\ UX^\top & UU^\top) = \mqty(X\\ U)\mqty(X\\ U)^\top.
      \end{split}
    \end{equation}
    Noting that $\Hess f(A,B)$ is a positive semidefinite symmetric matrix, and is positive definite if and only if it further satisfies $\rank\qty(\Hess f(A,B)) = n + k$. Using the fact that $\rank(C) = \rank(CC^\top)$ for any real-valued matrix, one can conclude that $\rank\mqty(X \\ U) = n+k$ is the necessary and sufficient condition for the uniqueness.
  \end{proof}

  \begin{rem}
    The minimizer of the optimization problem (\ref{eq:optimization:optimization:problem}) with fixed $D$ is not unique in a single-trajectory case if $m < n+k$, and the minimizer is unique if and only if the data matrix (i.e., $X$ for DMD, $\mqty(X & U)$ for DMDc) has no left null-space. This means that, in many applications, the DMD and DMDc algorithm easily results in unphysical coefficients, especially if the number of time steps are small. 

    On the other hand, for the same optimization problem with $N>1$, the Hessian can be positive definite even if the minimum number of time steps $\min_{\mu}(m_{\mu})$ is smaller than the combined dimension $n+k$ provided we have appropriate set of trajectories.
  \end{rem}
\fi

\printbibliography

@book{goodwin2001control,
  title={Control system design},
  author={Goodwin, G. C. and Graebe, S. F. and Salgado, M. E. and others},
  volume={240},
  year={2001},
  publisher={Prentice Hall Upper Saddle River}
}

@article{aastrom1971system,
  title={System identification—a survey},
  author={{\AA}str{\"o}m, K. J. and Eykhoff, P.},
  journal={Automatica},
  volume={7},
  number={2},
  pages={123--162},
  year={1971},
  publisher={Elsevier}
}

@book{beck2017first,
  title={First-order methods in optimization},
  author={Beck, A.},
  year={2017},
  publisher={SIAM}
}

@book{bertsekas2016nonlinear,
  title={Nonlinear Programming: 3rd Edition},
  author={Bertsekas, D.},
  year={2016},
  publisher={Athena Scientific}
}

@book{roberts1974convex,
  title={Convex Functions: Convex Functions},
  author={Roberts, A. W. and Varberg, D. E.},
  year={1974},
  publisher={Academic Press}
}

@book{nesterov2004introductory,
  title={Introductory Lectures on Convex Optimization: A Basic Course},
  author={Nesterov, Yurii},
  publisher={Kluwer Academic Publishers},
  year={2004},
  isbn={978-1-4020-7786-9}
}

@article{sato2024nearest,
author = {Sato, K. and Suzuki, M.},
title = {The nearest graph Laplacian in Frobenius norm},
journal = {Numerical Linear Algebra with Applications},
pages = {e2578},
year={2024}
}

@article{anzaki2024multi,
  title={Multi-trajectory Dynamic Mode Decomposition},
  author={Anzaki, R. and Yamada, S. and Tsutsui, T. and Matsuzawa, T.},
  year={2024},
  journal={Jxiv preprint}
}

@article{schmid2022dynamic,
author = {Schmid, P. J.},
title = {Dynamic Mode Decomposition and Its Variants},
journal = {Annual Review of Fluid Mechanics},
volume = {54},
number = {1},
pages = {225-254},
year = {2022},
}

@article{proctor2016dynamic,
  title={Dynamic mode decomposition with control},
  author={Proctor, J. L. and Brunton, S. L. and Kutz, J. N.},
  journal={SIAM Journal on Applied Dynamical Systems},
  volume={15},
  number={1},
  pages={142--161},
  year={2016},
  publisher={SIAM}
}

@article{hirsh2020centering,
author = {Hirsh, S. M. and Harris, K. D. and Kutz, J. N. and Brunton, B. W.},
title = {Centering Data Improves the Dynamic Mode Decomposition},
journal = {SIAM Journal on Applied Dynamical Systems},
volume = {19},
number = {3},
pages = {1920-1955},
year = {2020},

}

@article{anzaki2023dynamic,
  title={Dynamic mode decomposition with memory},
  author={Anzaki, R. and Sano, K. and Tsutsui, T. and Kazui, M. and Matsuzawa, T.},
  journal={Physical Review E},
  volume={108},
  number={3},
  pages={034216},
  year={2023},
  publisher={APS}
}

@article{tu2014dynamic,
title = {On dynamic mode decomposition:  Theory and applications},
journal = {Journal of Computational Dynamics},
volume = {1},
number = {2},
pages = {391-421},
year = {2014},
author = {J. H. Tu and C. W. Rowley and D. M. Luchtenburg and S. L. Brunton and J. N. Kutz},
}

@article{schmid2010dynamic,
  title={Dynamic mode decomposition of numerical and experimental data},
  author={Schmid, P. J.},
  journal={Journal of fluid mechanics},
  volume={656},
  pages={5--28},
  year={2010},
  publisher={Cambridge University Press}
}

@article{svenkeson2016spectral,
  title = {Spectral decomposition of nonlinear systems with memory},
  author = {Svenkeson, Adam and Glaz, Bryan and Stanton, Samuel and West, Bruce J.},
  journal = {Phys. Rev. E},
  volume = {93},
  issue = {2},
  pages = {022211},
  numpages = {10},
  year = {2016},
  publisher = {American Physical Society},
  %doi = {10.1103/PhysRevE.93.022211},
}


\end{document}